\documentclass[11pt]{amsart}
\usepackage{latexsym, amsmath, amssymb,amsthm,amsopn,amsfonts}
\usepackage{version}
\usepackage{epsfig,graphics,color,graphicx,graphpap}
\usepackage{amssymb}

\ifx\pdfoutput\undefined
  \DeclareGraphicsExtensions{.eps}
\else
  \ifx\pdfoutput\relax
    \DeclareGraphicsExtensions{.eps}
  \else
    \ifnum\pdfoutput>0
      \DeclareGraphicsExtensions{.pdf}
    \else
      \DeclareGraphicsExtensions{.eps}
    \fi
  \fi
\fi

\newcommand{\CI}{{\mathcal C}^\infty }

\newcommand{\R}{{\mathbb R}}

\newcommand{\HH}{{\mathcal H}}
\newcommand{\N}{{\mathbb N}}

\newcommand{\spec}{\operatorname{spec}}

\theoremstyle{plain}

\newtheorem{thm}{Theorem}
\newtheorem{prop}{Proposition}[section]

\newtheorem{lem}[prop]{Lemma}

\theoremstyle{definition}

\newtheorem{remk}{Remark}[section]

\numberwithin{equation}{section}

\def\squarebox#1{\hbox to #1{\hfill\vbox to #1{\vfill}}}

\newcommand{\la}{\langle}
\newcommand{\ra}{\rangle}

\newcommand{\und}{\frac{1}{2}}

\newcommand{\dom}{\mbox{dom}}
\newcommand{\Kcal}{{\mathcal{K}}}
\newcommand{\Dcal}{{\mathcal D}} 
\newcommand{\Lcal}{{\mathcal L}} 
\newcommand{\id}{\mbox{id}}
\usepackage{amsxtra}

\title
{Finite deficiency indices and uniform remainder in Weyl's law}

\author[L. Hillairet]
{Luc Hillairet}
\email{Luc.Hillairet@math.univ-nantes.fr}
\address{UMR CNRS 6629-Universit\'{e} de Nantes, 2 rue de la Houssini\`{e}re, \\
BP 92 208, F-44 322 Nantes Cedex 3, France}

\begin{document}
\maketitle
\section*{Introduction}
Von Neumann theory classifies all self-adjoint extensions 
of a given symmetric operator $A$ in terms of the so-called {\em deficiency indices}.
Since the choice of the self-adjoint condition reflects the physics that is underlying 
the problem, it is natural to ask how the spectra of two different self-adjoint 
extensions $A_0$ and $A_1$ can differ. 

In some cases, the deficiency indices are finite. 
This happens for instance in the following interesting settings from mathematical physics 
\begin{itemize}
\item[-] Quantum graphs (see \cite{BE, HK} and section \ref{qg} below),
\item[-] Pseudo-Laplacians, \v{S}eba billiards (see \cite{CdV, KMW}, and section \ref{PL} below),
\item[-] Manifolds with conical singularities (see \cite{Kokotov, KLP}),   
\item[-] Hybrid manifolds (see \cite{PRY}).
\end{itemize}

In any of these settings we prove the following theorem (see section \ref{secthm} 
and the applications for a more precise version)

\begin{thm}
Let $A$ be the symmetric operator associated with one of the 
preceding settings. There exists a constant $C$ such that 
for any self-adjoint extensions $A_0$ and $A_1$ of $A$ we have 
\[
\forall E,~~ \left| N_1(E)-N_0(E) \right | \leq C
\]
where $N_i$ denotes the spectral counting function of $A_i.$
\end{thm}

This fact actually derives from \cite{BS} ch. 9 sec. 3
\footnote{We are grateful to Alexander Pushnitski for indicating to us this reference.}. 
Our proof is slightly different and based on the min-max principle but 
the underlying ideas are similar.  
 
Motivation for this result came principally from \cite{KMW} and the \v{S}eba billiard setting. 
In this case we can take $A_0$ to be the standard 
Dirichlet Laplace operator in $R$, and this theorem proves that the remainder 
in Weyl's law is,  up to a $O(1)$ term, uniform with respect to the location of 
the Delta potential. In the case of one delta potential the uniform bound can also be derived 
from the fact that the spectra of the pseudo-laplacian and the usual laplacian are 
interlaced (see \cite{CdV} for instance).  

In constrast with \cite{BE, HK, KLP, PRY} we consider a rather 
crude spectral invariant. Moreover, our result relies on the min-max principle only 
which is less sophisticated than the analysis performed in the former references. 
It should be noted, however, that our result is not a straightforward 
byproduct of these results and should more likely be considered as a first step. 
From our perspective, it is quite interesting to have a general method 
allowing to get quite good hold on the spectral counting function before moving 
on to more complete spectral invariants such as heat kernel or resolvent estimates.

\textsc{Acknowledgments :} 
We are grateful to Alexander Pushnitski, Alexey Kokotov and Jens Marklof for 
useful comments on the first version of this note that resulted, in particular, 
in great improvment of the bibliography.
 
\section{Setting and Notations}
We begin by recalling some basic facts from spectral theory of self-adjoint 
operators as well as Von Neumann theory of self-adjoint extensions 
of a symmetric operator. We will use \cite{RS1,RS2} and \cite{CH} 
as references. 

\subsection{Basic Spectral Theory}
We consider a Hilbert space $\HH$ with scalar product 
$\langle \cdot ,\cdot \rangle$ and associated norm $\|\cdot\|.$ 

On $\HH$ we consider a symmetric operator $A$ with domain $\dom(A)$ and its adjoint $A^*.$ 
The graph norm is defined on $\dom(A)$ by $\|u\|_A^2\,=\,\|u\|^2+\|Au\|^2.$

An operator is self-adjoint if $A=A^*$. 
It has compact resolvent if the injection from $\dom(A)$ into $\HH$ 
is compact.  

The spectrum of a self-adjoint operator with compact resolvent consists 
in eigenvalues of finite multiplicities, that form a discrete set in $\R.$ 
There exists an orthonormal basis consisting of eigenvectors. 

If there exists $C\in \R$ such that $\forall u\in \dom(A),~ 
\|Au\|\geq C\|u\|,$ the operator is called {\em semibounded}.

For a semibounded self-adjoint operator with compact resolvent, 
the spectrum can be ordered into a non-decreasing sequence 
$(\lambda_n)_{n\in \N}.$

The spectral counting function is then defined by 
\[ 
N(E)\,=\, \mbox{Card}\{ \lambda_n\leq E\},\]
and by Courant-Hilbert min-max principle (see \cite{CH}) we have 

\begin{equation}\label{minmax}
\lambda_n\,=\, \min \left \{ \max\left \{ \frac{\la Au,u\ra}{\|u\|^2},~u\in F\backslash\{0\}\, \right \},
~ F\subset \HH, F ~\mbox{vector space s.t. }\dim F=n \right \}.
\end{equation}

\subsection{Von Neumann Theory} \hfill \\
This section summarizes section X.1 of \cite{RS2}. 
We define 
\begin{gather*} 
\Kcal^{\pm}\,=\,\ker(A^*\mp \id) \\
d_\pm\,=\,\dim ( \Kcal^\pm).
\end{gather*}

We recall that we have the following decomposition of $\dom(A^*)$ 
(see Lemma in section X.1 of \cite{RS2}).

\[ 
\dom(A^*)\,=\, \dom(\bar{A})\oplus \Kcal^+ \oplus \Kcal^-,
\]
and that the decomposition is orthogonal with respect to the scalar product 
$\langle \cdot,\cdot \rangle_{A^*}$ defined on $\dom(A^*)$ by 
\[ 
\forall u,v\in \dom(A^*),~ ~\langle u,v \rangle_{A^*}\,:=\,
\langle u,v \rangle+\langle A^*u,A^*v \rangle.
\]
We denote by $\pi_\pm$ the orthogonal projection from $\dom(A^*)$ onto 
$\Kcal^\pm$ and by $\pi_0$ the orthogonal projection onto $\dom(\bar{A})$ 
(so that $\pi_0\,=\,id-\pi_+-\pi_-).$

The following theorem (theorem X.2 and corollary of \cite{RS2}) 
provides a parameterization of all self-adjoint extensions of $A$
\begin{thm}\label{VN}
$A$ admits self-adjoint extensions if and only if $d_+\,=\,d_-.$\hfill \\ 
For any self-adjoint extension $A_{sa}$ of $A$, 
there exists a unique isometry $U$ from 
$\Kcal^+$ onto $\Kcal^-$ such that $\forall u\in \dom(A_{sa})$, 
$U\pi_+(u)\,=\,\pi_-(u).$
\end{thm}   

\section{The theorem}\label{secthm}
Using the notations of the preceding section we have 
\begin{thm}
Let $A$ be a symmetric operator with equal finite deficiency indices : 
\[ d_+\,=\,d_-= d<\infty.
\]
Suppose that there exists $A_0$ a self-adjoint extension 
of $A$ such that 
\begin{itemize}
\item[(i)] $A_0$ has compact resolvent,
\item[(ii)] $A_0$ is semibounded
\end{itemize}
Then 
\begin{enumerate}
\item Any other self-adjoint extension also has compact resolvent and 
is semibounded. 
\item There exist $E_0$ such that, for any other 
self-adjoint extension $A_1$ the following holds 
\[ 
\forall E\in \R,~~~ \left | N_1(E)-N_0(E) \right | \leq d
\]
where $N_i(E)$ denotes the spectral counting function of $A_i.$
\end{enumerate}
\end{thm}

As indicated in the introduction this result is actually already 
proved in \cite{BS} using that the difference of 
the resolvents $(A_0-z)^{-1}-(A_1-z)^{-1}$ is finite rank for some $z$.

\section{Proofs}
\subsection{$A_1$ has compact resolvent}
We prove this fact by proving the stronger lemma.
\begin{lem}
If $A_0$ has compact resolvent then the injection from $\dom(A^*)$ 
(equipped with $\|\cdot \|_{A^*}$) into $\HH$ is compact.
\end{lem}
\begin{proof}
Since $A_0$ has compact resolvent, the injection from $(\dom(A_0),\|\cdot \|_{A_0})$ into 
$\HH$ is compact. But $\dom(\bar{A}$) is closed in $\dom(A^*)$ for 
$\|\cdot\|_{A^*}$. Since $\|\cdot\|_{A_0}$ coincide 
with $\|\cdot \|_{A^*}$ on $A_0$, the injection 
from $(\dom(\bar{A}),\|\cdot\|_{A^*})$ is also compact.

Let $(u_n)_{n\in\N}$ be a sequence 
in $\dom(A^*)$ that is $\|\cdot\|_{A^*}$-bounded. 
We can extract a subsequence from $\pi_0 u_n$ that converges in $\HH$
since the injection from $(\dom(\bar{A}),\|\cdot\|_{A^*})$ into $\HH$ is compact.
On $\Kcal^\pm,$ $\|\cdot \|_{A^*}$ is equivalent to $\|\cdot \|$ and, since 
$d^\pm$ are finite we can also extract convergent subsequences. 
This proves the Lemma.
\end{proof}
  
The Lemma says that $A^*$ has compact resolvent and so does $A_1$  
since $A^*$ extends $A_1$ and the latter is closed.

\subsection{$A_1$ is semibounded}
(See also \cite{AG} sec. 85)  

Let $(\lambda_k(A_0))_{k\in \N}$ denote the (ordered spectrum) 
of $A_0$ and consider $n$ such that $\lambda_n(A_0)\geq 0.$ 

Consider $F\subset \dom(A_1)$of dimension $n+d.$ Denote by 
$F_1\,=\,F\cap \dom(\bar{A})= F \cap \ker(\id - \pi_0).$ 
Theorem \ref{VN} implies that $\ker (\id-\pi_0)_{|\dom(A_1)}$ 
is of dimension $d,$ and thus, since $F$ is of dimension $n+d,$
$F_1$ is of dimension at least $n.$

Moreover $F_1\subset \dom(\bar{A})\subset \dom (A_{1,0}),$ thus, 
for all $u\in F_1$ we have 
\[
\la A_1 u,u\ra\,=\,\la A_0 u,u\ra .\]
Since $F_1\subset F$ it follows that 
\[
\max_{u\in F, u\neq 0} \frac{\la A_1u,u\ra}{\| u\|^2}\geq
\max_{u\in F_1, u\neq 0} \frac{\la A_0 u,u\ra}{\| u\|^2}.   
\]
Since $\dim F_1\geq n$ and $\lambda_n\geq 0$, it follows from 
the min-max principle that the right-hand side is non-negative. 

We thus obtain, that for all $F\subset \dom(A_1)$ of dimension $n+d$ 
we have 
\[
\max_{u\in F, u\neq 0} \frac{\la A_1u,u\ra}{\| u\|^2}\geq 0.
\]
This implies that $A_1$ has at most $n+d-1$ negative eigenvalues. 
(otherwise the subspace generated by $n+d$ negative eigenvalues 
would contradict the preceding bound).  

\subsection{Comparing $N_0$ and $N_1$}
According to the previous section we have that $A_1$ also is 
semibounded with compact resolvent so that we can denote by $(\lambda_k(A_1))_{k\in \N}$ 
its ordered spectrum. We also denote by $V_n^i$ the vector space 
generated by the first $n$ eigenvectors of $A_i.$

For any $n,$ set $F\,=\, V_{n+d}^1\cap \dom (\bar{A}).$ 
Making the same reasoning as previously we find that $F$ is of dimension at least $n$ and 
\[
\max_{u\in V_{n+d}, u\neq 0}\left  \{\frac{\la A_1u,u\ra}{\| u\|^2}\right \}\,\geq\,
\max_{u\in F, u\neq 0} \left\{ \frac{\la A_0 u,u\ra}{\| u\|^2}\right \}.   
\] 
By definition of $V_{n+d}$ the left-hand-side is $\lambda_{n+d}(A_1)$ 
and, using the min-max principle, the right-hand side is bounded below 
by $\lambda_n(A_0).$
Thus, we obtain, 
\[
\forall n,~~~\lambda_{n+d}(A_1)\geq \lambda_n(A_0).
\]
Observe that $N_0(E)$ is characterized by 
\[
\lambda_{N_0(E)}(A_0)\leq E < \lambda_{N_0(E)+1}(A_0)
\]
Using this and the preceding inequality we find that, 
for all $E$ we have 
\[ 
E\leq \lambda_{N_0(E)+d+1}(A_1)
\]
An thus, for all $E$ we have 
\[
N_1(E)\leq N_0(E)+d.
\]

Since $A_0$ and $A_1$ now play symmetric roles we also have 
\[
\forall n, ~ ~ ~ \lambda_{n+d}(A_0)\geq \lambda_n(A_1).
\]
and thus 
\[ N_0(E)\leq N_1(E)+d.\]
The final claim of the theorem follows.

\section{Applications}
\subsection{Quantum graphs}\label{qg}
Quantum graphs are now well-studied objects from mathematical 
physics (see \cite{Kuchment} for an introduction). 
A very rough way of defining a (finite) quantum graph is the following. 

Pick $K$ positive real numbers (the lengths), set 
$\HH= \bigoplus_{i\,=\,1}^{K} L^2(0,L_i)$ and $\Dcal\,=\, 
\bigoplus_{i\,=\,1}^{K} \CI_0(0,L_i)$ and define $A$ on $\Dcal$ 
by $ A(u_1\oplus u_2 \oplus \cdots \oplus u_K)\,=\,-(u_1''\oplus u_2''\oplus\cdots \oplus u_K'').$
This operator is symmetric. Any self-adjoint extension of $A$ is called a quantum graph. 

\begin{remk}
Usually quantum graphs are constructed starting from 
a combinatorial graph given by its edges and vertices. This combinatoric 
data is actually hidden in the choice of the self-adjoint condition.
\end{remk}

One basic question is to understand to which extent the knowledge of the spectrum 
determines the quantum graph (i.e. the lengths and the boundary condition). 

It is known that there are isospectral quantum graphs \cite{BPBS} and the 
following theorem says that, as far as counting 
function is concerned it is quite difficult to 
determine the self-adjoint condition.

\begin{thm}
For any quantum graph with $K$ edges the following 
bound holds :
\[ 
\left| N(E)- \frac{\Lcal}{\pi}E_+^\und \right | \leq 3K
\]
where $\Lcal:=\sum_{i=1}^K L_i$ and $E_+\,:=\max(E,0).$
\end{thm}

\begin{proof}
Fix $K$ and the choice of the lengths. 
We choose one particular self-adjoint extension $A_D$ 
that consists in decoupling all the edges and putting Dirichlet boundary 
condition on each end of each edge. 
The spectrum is then easily computed and we have 
\[ 
\spec(A_D)\,=\, \bigcup_{i=1}^K \left\{ \frac{k^2\pi^2}{L_i^2},~ k\in \N \right\}.
\]
In particular we have (denoting by $N_D$ the counting function of the Dirichlet extension)
\[ 
N_D(E)\,=\,\sum_{i=1}^K \left[ \frac{L_i}{\pi}E_+^\und \right ],
\]
where $[\cdot]$ denotes the integer part. 
In particular, we have 
\[
\left | N_D(E)- \frac{\Lcal}{\pi}E_+^\und \right | \,\leq K.
\]
We now compute the deficiency indices. A straightforward computation 
yields $d_+\,=\,d_-\,=\,2K.$
And thus, using the main theorem and triangular inegality 
we obtain that for any quantum graph
\[
\left | N_D(E)- \frac{\Lcal}{\pi}E_+^\und \right | \,\leq 3K,
\]     
indepently of the choice of lengths.
\end{proof}

\subsection{Pseudo-Laplacian with Delta potentials}\label{PL}
It is known that on Riemannian manifolds of dimension $2$ or $3$ 
it is possible to add so-called Delta potentials. 
From a spectral point of view this corresponds to choosing a finite set of points $P$ 
and to consider the Riemannian Laplace operator 
defined on smooth functions with support in $M\backslash P.$ (see \cite{CdV} for instance)

\begin{remk}
This construction is also possible starting from a bounded domain in $\R^2$ with, say, 
Dirichlet boundary condition. We obtain the so-called \v{S}eba billiards 
(See \cite{KMW} for instance)
\end{remk} 

There are several self-adjoint extensions and (a slight 
generalization of) Lemma of \cite{CdV} proves that, 
in this setting the deficiency indices are $d:=\mbox{card P}$.
Following Colin de Verdi\`ere we call any such self-adjoint extension 
a {\em Pseudo-Laplacian with $d$ Delta potentials}

Application of the theorem gives the following.

\begin{thm}
Let $M$ be a closed Riemannian manifold of dimension $2$ or $3.$
Let $N_0$ be the counting function of the (standard) Laplace operator 
on $M$. For any pseudo-laplacian with $d$ Delta potential the following 
holds 
\[
\left| N(E)-N_0(E) \right | \leq d
\]
\end{thm}

It should be noted first 
that the bound depends only on the number of Delta potentials 
and not on their location, and second that the effect of adding 
Delta potentials is much smaller than the usual known remainder 
terms in Weyl's law for $N_0(E).$     

\subsection{Others}
There are two other settings where finite deficiency indices 
occur that are worth mentioning. In both case one could apply the 
theorem to get a Weyl's asymptotic formula independent 
of the choice of the self-adjoint condition up to a $O(1)$ term. 
These are 
\begin{enumerate}
\item Manifolds with conical singularities. The common 
self-adjoint extension in use corresponds to Friedrichs extension and 
if one changes the self-adjoint extensions at the conical points 
(see \cite{Kokotov,KLP}) then the counting function is affected only by 
some bounded correction. Observe that the deficiency index associated with 
the Laplace operator on the cone of opening angle $\alpha$ is 
$2[\frac{\alpha}{2\pi}]-1$ (with $[\cdot]$ the integer part).

\item The so-called {\em hybrid manifolds} that are studied in 
\cite{PRY} which are obtained by, in some sense, grafting  
quantum graphs onto higher dimensional manifolds. Here again, 
one can compare the counting function of the chosen self-adjoint extension 
to the natural one which is obtained when all the parts are decoupled. 
\end{enumerate}

\end{document}